\documentclass[12pt]{article}
\usepackage{amsfonts,amsmath,amsthm,amssymb,enumerate,array}
\usepackage{geometry}
\usepackage{hyperref} 
\usepackage{textcomp}
\usepackage{graphicx}
\usepackage{subfig}
\usepackage{float}
\usepackage[utf8]{inputenc}
\usepackage{multirow}
\usepackage{lscape}
\usepackage{upgreek}
\usepackage{ wasysym }
\usepackage{ marvosym }
\usepackage[all]{xy}
\usepackage[usenames]{color}
\usepackage[utf8]{inputenc}

\newtheorem{theorem}{Theorem}[section] 

\newtheorem{lema}{Lemma}[section]

\numberwithin{equation}{section} 
\title{Isoparametric functions on $\mathbb{R}^n\times\mathbb{M}^m$}
\author{Jurgen Julio-Batalla\footnote{E-mail address: jurgen.julio@cimat.mx}}
\date{}

\begin{document}

\maketitle

\begin{abstract}
We classify the isoparametric functions on $\mathbb{R}^n\times\mathbb{M}^m$, $n, m\geq2$, with compact level sets,  where $\mathbb{M}^m$ is a connected, closed Riemannian manifold of dimension $m$. Also, we classify the isoparametric hypersurfaces in $\mathbb{S}^2\times\mathbb{R}^2$ with constant principal curvatures. 
\end{abstract}


\section{Introduction}
Let $(N,h)$ be a connected Riemannian manifold. A non-constant smooth function $f:N\rightarrow\mathbb{R}$ is called \textit{isoparametric} if there exist smooth functions $a,b:\mathbb{R}\rightarrow\mathbb{R}$ such that
$$(1)\;|\nabla f|^2=a(f)\quad \text{and}\quad (2)\;\Delta f=b(f).$$
The smooth hypersurfaces $M_t=f^{-1}(t)$ for $t$ regular value of $f$ are called \textit{isoparametric hypersurfaces}. The condition $(1)$ means that the hypersurfaces $M_t$ are parallel and the condition $(2)$ says that these hypersurfaces have constant mean curvatures. The preimage of the maximum and minimum  of the isoparametric function $f$ are denoted by $M_+$ and $M_-$ (resp.), they are called \textit{focal varieties} of $f$.\\
The problem of classification of the isoparametric hypersurfaces in Riemannian manifolds started with the works of É. Cartan who proved in \cite{2} that, when the ambient manifolds has constant curvature, a hypersurface is isoparametric if and only if has constant principal curvatures. A complete classification in Euclidean and real hyperbolic spaces followed,
 but the case of the spheres was richer much more difficult, and only recently a complete classification was obtained \cite{Chi}. Cartan classified the isoparametric hypersurfaces in the sphere with $l\in\{1,2,3\}$ different principal curvatures. Later, Münzner \cite{Munzner} proved that, an isoparametric hypersurface in $\mathbb{S}^n\subset\mathbb{R}^{n+1}$ with $l$ distinct principal curvatures is contained in a level set of a homogeneous polynomial of degree $l$ on $\mathbb{R}^{n+1}$ satisfying certain equations  known now as the Cartan–Münzner differential equations. He used this to prove  that the number $l$ of distinct principal curvatures  can only be 1; 2; 3; 4, or 6. Then several authors worked on the difficult cases of $l=4$ or 6 distinct principal curvatures: see
for instance \cite{ChiIII, Cecil, Miyaoka, Chi}.
For a more detailed study of  isoparametric hypersurfaces on space forms see for instance \cite{3}. \\
The study of isoparametric functions on general Riemannian manifolds started with the work of Q. M. Wang in
\cite{8}. Many interesting results have been obtained recently about isoparametric hypersurfaces on 
different spaces \cite{Exotic1, Exotic2, Thor, 6, 7}.\\ 
Isoparametric hypersurfaces allow to reduce  certain systems partial differential equations to ordinary differential equations, which can help to find explicit solutions. This is one of the reasons for which it  is important to investigate exist of such functions.\\ 
For instance it has been applied to the study of multiplicity of solutions to the Yamabe problem on the Riemannian manifold $(M^m,g)$ (see \cite{Henry}), which consists of finding metrics of constant scalar curvature conformal to $g$. If the scalar curvature of $g$ (denote $s_g$) is constant, then
writing a conformal metric as $h=u^{\frac{4}{m-2} } g$ (for a positive function $u$) we have that $h$ has constant scalar curvature $\lambda$ if and
only if $u$ is a positive solution of  the Yamabe equation $$-\frac{4(m-1)}{m-2}\Delta_gu+s_gu=\lambda u^{\frac{m+2}{m-2}}.$$
If there is an isoparametric function $f$ on $(M,g)$ then one can look  for solutions of the form $u = \varphi \circ f$. It follows that $u$ solves the Yamabe equation if $\varphi$ solves the ordinary differential equation $$-\frac{4(m-1)}{m-2}(\varphi^{\prime\prime}a+\varphi^{\prime}b)+ s_g\varphi=\lambda\varphi^{\frac{m+2}{m-2}}\quad \text{for}\; a,b\; \text{given by}\; (1),(2).$$
In the case of a Riemannian product $(M \times N , g+h )$ an isoparametric function on any of the factors gives an isoparametric function in the product. However, there are examples 
of products with mixed isoparametric functions. The trivial examples are the radial functions on $\mathbb{R}^n$, but there are also examples for instance in $\mathbb{S}^2\times\mathbb{S}^2$ (see \cite{6}).\\ 
We will be interested in this article in products with Euclidean space, namely $(M\times \mathbb{R}^n , g + dx^2 )$. It is  an important problem  to understand the finite energy solutions to the Yamabe equation  in such products (see for instance \cite{Akutagawa, Ammann}). Note that positive finite energy  solutions (which have to vanish at infinity)  
must have compact level sets.  There is a well-known such solution 
which is a radial function on $\mathbb{R}^n$ (see \cite{Akutagawa}). Are there other solutions? It is conjectured that the answer is NO under certain conditions, for instance if $g$ is Einstein. The case when
$M=\mathbb{S}^m$ is particularly important. For instance, when $n=m=2$ if there exist no such solution then one would prove that the Yamabe invariant of
$\mathbb{S}^2 \times \mathbb{S}^2$ is strictly greater than the one of $\mathbb{CP}^2$.\\ 
Our first result says that such solutions could not be built by an isoparametric function: 
\begin{theorem}\label{a}
An isoparametric function on $\mathbb{R}^n\times\mathbb{ M}^m$, $n,m \geq 2$,  with compact level sets $($where $\mathbb{M}^m$ is a closed Riemannian manifold$)$ is a radial function of $\mathbb{R}^n$.
\end{theorem}
Another interesting fact is related with rigidity of gradient Ricci solitons (see \cite{5, Fernandez}). A gradient Ricci soliton is a Riemannian manifold $(M,g)$ together a smooth function $f$ that satisfies $Ric+Hessf=\lambda g$, where $\lambda$ is a real constant, $Hessf$ is the Hessian of the function $f$ (which is called the potential of the soliton). Taking a product $N\times\mathbb{R}^k$ with $N$ being Einstein with Einstein constant $\lambda$ and $f =\lambda |x|^2/2$ on $\mathbb{R}^k$ yields a gradient Ricci soliton. We say that a gradient soliton is \textit{rigid} if it is isometric to a quotient of $N\times\mathbb{R}^k$, where $N,\mathbb{R}^k$ and $f$ as above. Now, if $(M,g,f)$ is a gradient Ricci soliton for $\lambda\neq 0$ then $\nabla(s_g +|\nabla f|^2)=2\lambda\nabla f$ (see \cite{5}) and $\Delta f=n\lambda-s_g$: so if $s_g$ is constant then
$f$ is an isoparametric function. Our first result also implies that, a gradient Ricci soliton of constant scalar curvature of
the form $(N\times\mathbb{R}^k,h+dx^2,f)$  with $N$ compact and such that  each level set of $f$ is also compact, is rigid (but this can also be proved directly without using the theorem).\medskip \\
Without the compactness condition on the level sets of the isoparametric function one would still like to know if there could be examples which do not come
from isoparametric functions on $M$. We will only consider the case of $\mathbb{S}^2 \times\mathbb{R}^2$. Using the ideas developed by Urbano in \cite{6} for the case $\mathbb{S}^2\times\mathbb{S}^2$,  we will prove 

\begin{theorem}\label{b}
The isoparametric hypersurfaces with constant principal curvatures in $\mathbb{S}^2\times\mathbb{R}^2$ are of the form $\mathbb{S}^2\times\mathbb{S}^1(r)$ $($for $r\in\mathbb{R}^+)$ or $\mathbb{S}^1(t)\times\mathbb{R}^2$ $($for $t\in(0,1))$.
\end{theorem}
Note that in general, there are examples of isoparametric hypersurfaces with  non-constant principal curvatures, 
as in the examples  in \cite{7} 
for certain  complex projective spaces.\medskip \\
\textbf{Acknowledgements}. The author would like to thank Jimmy Petean for many useful comments and support during preparation of this paper.

\section{Compact isoparametric hypersurfaces in \texorpdfstring{$\mathbb{R}^n\times\mathbb{M}^m$}{} }
In this section we will prove Theorem \ref{a}.
We start by recalling some structural results for isoparametric functions on general Riemannian manifolds.\medskip \\
Let $f$ be an isoparametric function on a connected complete Riemannian manifold $L$. Then
\begin{enumerate}
\item The focal varieties of $f$ are smooth submanifolds of $L$ (Theorem A, \cite{8});
\item The interior of $f(L)$ only has regular values (Lemma 3, \cite{8});
\item Each regular level set of $f$ is a tube over either of the focal varieties (Theorem A, \cite{8});
\item The focal varieties are pure minimal submanifolds (Theorem 1.3, \cite{4}).
\end{enumerate}

Now, Let $\mathbb{M}^m$ be a connected, closed Riemannian manifold. We consider a family of compact isoparametric hypersurfaces $M_t$  in $\mathbb{R}^n\times\mathbb{M}^m$ with $n,m\geq 2$, i.e. exist an isoparametric function $f:\mathbb{R}^n\times\mathbb{M}^m\rightarrow\mathbb{R}$ such that each $M_t=f^{-1}(t)$ is compact.\medskip \\
If the focal varieties of $f$ are empty (i.e. $M_-=M_+=\phi$) then, from  Theorem 1.1 in \cite{4}  and the fact that $\mathbb{R}^n\times\mathbb{M}^m$ can not be an $\mathbb{S}^1$ bundle over some $M_t$ (since each $M_t$ are compact), $\mathbb{R}^n\times\mathbb{M}^m$ is a rank one vector  bundle over some $M_t$ regular hypersurface. It is well-known that exits a deformation retract of the total space over the base space of a vector bundle. This implies that, the homology group of $M_t$ and $\mathbb{R}^n\times\mathbb{M}^m$ are equivalent. In particular, we obtain  $0=H_{m+n-1}(\mathbb{M}^m)=H_{m+n-1}(M_t)$ since $n-1>0$, which is a contradiction. Therefore there is a non-empty
focal variety.\\
In the case that $f$ has $M_-\neq\phi$ and $M_+\neq\phi$, again by Theorem 1.1 in \cite{4} we have that $\mathbb{M}^m\times\mathbb{R}^n$ is diffeomorphic to a union of two disk bundles over $M_+$ and $M_-$. Since that $M_-$ and $M_+$ are compact, $\mathbb{R}^n\times\mathbb{M}^m$ would be compact.\\
Without loss of generality, we can assume that the set $M_-$ of minimum points of $f$ it is non-empty and $M_+ = \phi$. \\ 
Since $\mathbb{R}^n\times\mathbb{M}^m$ cannot be the union of two disk bundles over compact submanifolds, we have that $\mathbb{R}^n\times\mathbb{M}^m$ is a vector bundle over $M_-$ (Theorem 2 in \cite{1}).\medskip \\
Now, we point out  a  fact about minimal submanifolds. See \cite{9} for more details.
\begin{lema}
Let $\Phi:L^n\rightarrow\mathbb{R}^k$ be an isometric immersion with the mean curvature vector $H$, then $$\Delta\Phi=nH,$$
where $\Delta\Phi=(\Delta\Phi^1,\cdots,\Delta\Phi^k).$
\end{lema}
\begin{proof}
Let $\{e_i\}$ be a local orthonormal frame field of $L$. Then
\begin{align*}
\Delta\Phi&=\sum_{i}\nabla^{\mathbb{R}^k}_{\Phi_*e_i}\Phi_*e_i-\Phi_*\nabla^{L}_{e_i}e_i\\
&=\sum_{i} (\nabla^{\mathbb{R}^k}_{\Phi_*e_i}\Phi_*e_i)^{\perp}=nH.
\end{align*}
\end{proof}
On the other hand, let $L\rightarrow\bar{L}\subset\bar{\bar{L}}$ be isometric immersions with connections $\nabla,\bar{\nabla}$ and $\bar{\bar{\nabla}}$ respectively. Denote $H$ and $\bar{H}$ to be the mean curvatures of $L$ in $\bar{L}$ and $L$ in $\bar{\bar{L}}$ respectively. Then
\begin{align*}
nH&=\sum(\bar{\nabla}_{e_i}e_i)^{\perp}\\
&=(\sum(\bar{\bar{\nabla}}_{e_i}e_i)^{T\bar{L}})^{\perp}\\
&=(\sum(\bar{\bar{\nabla}}_{e_i}e_i)^{\perp})^{T\bar{L}}=n\bar{H}^{T\bar{L}}.
\end{align*}
In our situation, $\Phi:M_-\rightarrow\mathbb{R}^n\times\mathbb{M}^m\subset\mathbb{R}^{n+k}
$ is minimal for some $k$. Thus
\begin{equation*}
\Delta\Phi\perp T(\mathbb{R}^n\times\mathbb{M}^m).
\end{equation*}
Hence, $(\Delta\Phi^1,\cdots,\Delta\Phi^n)=0$.\\
Since $M_-$ is compact it follows that the functions $\Phi^j$   are constant for all $j=1,\ldots,n$. Thus, the focal variety of $f$ is of the form $M_-=\{p\}\times V$, where $V\subseteq\mathbb{M}^m$ is a submanifold and $p\in\mathbb{R}^n$.\medskip \\
Since that  the submanifold $V$ and $\mathbb{R}^n\times\mathbb{M}^m$ are homotopy equivalent, we have
$$H_m(\mathbb{M}^m)=H_m(V).$$
Therefore, $dim(V)=m$ and $$M_-=\{p\}\times\mathbb{M}^m.$$
But since the level sets $M_t$ are tubes over $M_-$ this of course implies Theorem \ref{a}.

\section{Isoparametric hypersurfaces with constant principal curvatures in \texorpdfstring{$\mathbb{S}^2\times\mathbb{R}^2$}{}}

In this section we will prove Theorem \ref{b}. Denotation and  background will be the same as in \cite{6} and we refer the reader to this article for more details.\\
Let $\mathbb{S}^2$, $\mathbb{R}^2$  be  space forms with curvatures 1 and 0 respectively. We define the complex structures $L_1$ and $L_2$ by:
\begin{align*}
L_1:T\mathbb{S}^2&\rightarrow T \mathbb{S}^2\\
v&\mapsto L_1(v):=p \wedge v \quad \text{for}\; p\in\mathbb{S}^2,\quad v\in T_p\mathbb{S}^2;\\
L_2:\mathbb{R}^2&\rightarrow \mathbb{R}^2\\
(q_1,q_2)&\mapsto\;L_2((q_1,q_2)):=(-q_2,q_1).
\end{align*}
We consider $\mathbb{S}^2\times \mathbb{R}^2$ with the product metric and the complex structures $J_1=(L_1,L_2),$ $J_2=(L_1,-L_2)$. We notice that the product structure $P$ in  $\mathbb{S}^2\times  \mathbb{R}^2$ defined by $P(v_1,v_2)=(v_1,-v_2)$ satisfies that $P=-J_1J_2=-J_2J_1$, moreover, $P$ is parallel with respect of the Levi-Civita connection of $\mathbb{S}^2\times \mathbb{R}^2$.\medskip \\
Let $M^3\subset \mathbb{S}^2\times\mathbb{R}^2$ be an oriented hypersurface with $N=(N_1,N_2)$ a unit normal vector field to $M^3$. We consider the function $C$ and vector field $X$ tangent to $M^3$  given by:
$$C:=\langle PN,N\rangle\;\; \text{and}\;\; X:=PN-CN.$$
\begin{lema}
Let $f:\mathbb{S}^2\times\mathbb{R}^2\rightarrow\mathbb{R}$ be an isoparametric function. If each regular hypersurfaces $M_t=f^{-1}(t)$  has constant principal curvatures, then the function $C_t$ corresponding to each $M_t$ is constant.
\end{lema}
\begin{proof}
By condition (1) of the definition of isoparametric function the unit vector field $N=\frac{\nabla f}{|\nabla f|}$  is a geodesic field. Since  the product structure $P$ is parallel,  the function $C_t$ is independent of the regular hypersurfaces $M_t$ (since $N(C_t)=\langle\nabla_NN,PN\rangle+\langle N,P\nabla_NN\rangle=0$).\\
Now, we consider the open set $U=\left\lbrace p\in M_t/ C^2(p)<1 \right\rbrace$. If $U$ is not empty then we can consider on $U$ the local orthonormal frame field, 
$$B=\left\lbrace  B_1=\frac{X}{\sqrt{1-C^2}},B_2=\frac{J_1N+J_2N}{\sqrt{2(1+C)}},B_3=\frac{J_1N-J_2N}{\sqrt{2(1-C)}}\right\rbrace.$$
From the Radial Curvature Equation  we have
$$-\nabla_NS_t + S_t^2=-R_N,$$
where $S_t$ is the shape operator of $M_t$ corresponding to $N$ and $R_N(\cdot)=R(\cdot,N)N$.\\
Taking trace we obtain that
$$-tr(R_N)=-tr(\nabla_NS_t)+tr(S_t^2)=-\nabla_N trS_t+tr(S_t^2)=-3H^{\prime}(t)+tr(S_t^2). $$
We are assuming that the principal curvatures $\mu_1(t),\mu_2(t)$ and $\mu_3(t)$ of $M_t$ are constant, then we have that $tr(S_t^2)=\mu_1^2+\mu_2^2+\mu_3^2$ is constant. Therefore $tr(R_N)$ is constant in $M_t$.\\ 
Now we compute $tr(R_N)$  in the frame $B$
\begin{align*}
R_N(B_1)&=\frac{1}{4\sqrt{1-C^2}}R^{\mathbb{S}^2}(PX+X,PN+N)N\\
&=\frac{1}{4\sqrt{1-C^2}}R^{\mathbb{S}^2}(N+X-C(X+CN),N+X+CN)N\\
&=\frac{1}{4\sqrt{1-C^2}}\{R^{\mathbb{S}^2}(N+X,CN)N +R^{\mathbb{S}^2}(-C(X+CN),N)N\}\\
&=\frac{1}{4\sqrt{1-C^2}}\{R^{\mathbb{S}^2}(X,CN)N+R^{\mathbb{S}^2}(-C(X+CN),N)N\}=0;\\
R_N(B_3)&=R^{\mathbb{R}^2}(B_3,N_2)N=0;\\
R_N(B_2)&=\frac{J_1N+J_2N}{\sqrt{2(1+C)}}|N_1|^2.
\end{align*}
Thus $tr(R_N)= \frac{1+C}{2}$ and the Lemma follows.
\end{proof}
Theorem \ref{b} is equivalent to the following:\\
{\bf Claim:} The isoparametric functions on $\mathbb{S}^2\times\mathbb{R}^2$ with regular level sets of constant principal curvatures only depend on one factor, i.e. $C^2=1$.
\begin{proof}
Let $f$ be an isoparametric function on $\mathbb{S}^2\times\mathbb{R}^2$ with $M_t=f^{-1}(t)$ of constant principal curvatures. The Lemma 2.1 implies that the function $C$ is constant.\\
Assume that $C\in(-1,1)$.\\
We are going to express the shape operator $S_0=S$ and tangential component of the product structure $P^T$ in the orthonormal frame field
\begin{equation*}
B= \left\{ B_1=\frac{X}{\sqrt{1-C^2}},\; B_2=\frac{(J_1+J_2)N}{\sqrt{2(1+C)}},\; B_3=\frac{(J_1-J_2)N}{\sqrt{2(1-C)}} \right\} .
\end{equation*}
Note that, 
$$\langle\nabla C,Y\rangle=\nabla_Y \langle N,PN\rangle=\langle \nabla_YN,PN\rangle+\langle N,P\nabla_YN\rangle=2\langle PN,-S(Y)\rangle=\langle -2S(X),Y\rangle$$
for $Y\in\Gamma(M)$. Then $S(X)=-\nabla C/2=0$ and we can write
$$S=
\left(
\begin{array}{ccc}
0 & 0  & 0 \\
0 & \sigma_{22} & \sigma_{23}\\
0 & \sigma_{23} & \sigma_{33}
\end{array}
\right).
$$ 
On the other hand, 
$$PB_1=\frac{1}{\sqrt{1-C^2}}(-CX+N(1-C^2))\;,\; PB_2=B_2\;,\; PB_3=-B_3, $$
thus
$$P^T=
\left(
\begin{array}{ccc}
-C & 0 & 0 \\
0 & 1 & 0\\
0 & 0 &-1
\end{array}
\right).
$$
By a direct computation, we obtain that
$$\Delta C=-6\langle X,\nabla H\rangle-2tr(S^2)C+2tr(P^TS^2).$$
Then we have
$$tr(S^2)C=tr(P^TS^2)=\sigma^2_{22}-\sigma^2_{33}=3H(\sigma_{22}-\sigma_{33}).$$
Assume now that $H\neq 0$.\\
From the expression  $$\sigma_{22}-\sigma_{33}=\frac{C}{3H}tr(S^2),$$
we see that  $\sigma_{22}-\sigma_{33}$ must be  constant.\medskip \\
Then
$\sigma^2_{22}+\sigma^2_{23}$ is also constant since $tr(S^2)+3H\{\sigma_{22}-\sigma_{33}\}=2\sigma^2_{22}+2\sigma^2_{23}$.\\
And since $$tr(S^2)+9H^2=2\sigma^2_{22}+2\sigma^2_{23}+2\sigma^2_{33}+2\sigma_{22}\sigma_{33},$$ 
we have that $2\sigma_{33}(\sigma_{33}+\sigma_{22})$ is constant.\\
Since $H\neq 0$, $\sigma_{33}$ must be constant, and hence $\sigma_{22}$ is also constant. It follows that each $\sigma_{ij}$ is constant.\medskip \\
Now, we compute $X(\sigma_{22})$ and $X(\sigma_{33}).$\\
Before we remember  the Codazzi equation of $M$ and the Hessian of the function $C$ respectively:
\begin{align*}
\nabla S(V,W,Z)-\nabla S(W,V,Z)&=\frac{1}{4}\langle V,X\rangle \langle PW+W,Z\rangle-\frac{1}{4}\langle W,X\rangle \langle PV+V,Z\rangle;\\
\nabla^2C(V,W)&=-2\nabla S(V,X,W)-2C\langle SV,SW\rangle + 2\langle PSV,SW\rangle.
\end{align*}
Since $J_i$ are parallel, we have $\nabla_XB_j=0$ for $j=1,2,3.$\\
Thus,
\begin{align*}
X(\sigma_{22})&=\nabla S(X,B_2,B_2)\\
&=\nabla S(B_2,X,B_2)+ \frac{|X|^2}{2}\\
&=\frac{|X|^2}{2} +\langle PSB_2,SB_2\rangle -C\langle SB_2,SB_2\rangle\\
&=\frac{1-C^2}{2} + (1-C)\sigma^2_{22} - (1+C)\sigma^2_{23},\\\\
X(\sigma_{33})&=\nabla S(X,B_3,B_3)\\
&=\nabla S(B_3,X,B_3) \\
&=\langle PSB_3,SB_3\rangle -C\langle SB_3,SB_3\rangle\\
&=(\sigma^2_{23}-\sigma^2_{33}) - C(\sigma^2_{23}+\sigma^2_{33})\\
&=(1-C)\sigma^2_{23}-(1+C)\sigma^2_{33}.
\end{align*}
From the equation $(\sigma_{22}^2+2\sigma_{23}^2+\sigma_{33}^2)C=tr(S^2)C=3H(\sigma_{22}-\sigma_{33})=\sigma_{22}^2-\sigma_{33}^2$, we obtain
$$(C-1)\sigma^2_{22}+(C+1)\sigma^2_{33}=-2C\sigma^2_{23}.$$
By combining the last expression with $X(\sigma_{33})=0$, we get
$$(C-1)\sigma^2_{22}=-(1+C)\sigma^2_{23}.$$
But $X(\sigma_{22})=0$ allow us to get
$1-C^2=0$.\medskip \\
The above argument means that the family of isoparametric hypersurfaces with constant function $C\in(-1,1)$ are all minimal.\\
But, if we assume that the family $M_t$ are all minimal, then from the trace of the Radial Curvature Equation we have
$$0=3H^{\prime}(t)=tr(S_t^2)+tr(R_N)=tr(S_t^2)+\frac{1+C}{2}>0.$$
This argument allow us to conclude that $C^2=1$.
\end{proof}

\end{document}